\documentclass[11pt,a4paper]{amsart}
\usepackage{graphicx,multirow,array,amsmath,amssymb}

\newtheorem{theorem}{Theorem}

\newtheorem{lemma}[theorem]{Lemma}

\begin{document}

\title{Enumerating independent vertex sets in grid graphs}

\author[S. Oh]{Seungsang Oh}
\address{Department of Mathematics, Korea University, Seoul 02841, Korea}
\email{seungsang@korea.ac.kr}

\author[S. Lee]{Sangyop Lee}
\address{Department of Mathematics, Chung-Ang University, Seoul 06974, Korea} 
\address{Korea Institute for Advanced Study, Seoul 02455, Korea} 
\email{sylee@cau.ac.kr}

\thanks{The first author was supported by the National Research Foundation of Korea(NRF) grant funded
by the Korea government(MSIP) (No. NRF-2014R1A2A1A11050999).}
\thanks{The second author was supported by the National Research Foundation of Korea(NRF) grant funded
by the Korea government(MSIP) (No. NRF-2013R1A1A2A10064864).}

\begin{abstract}
A set of vertices in a graph is called independent 
if no two vertices of the set are connected by an edge.
In this paper we use the state matrix recursion algorithm, developed by Oh, 
to enumerate independent vertex sets in a grid graph 
and even further to provide the generating function with respect to the number of vertices.
We also enumerate bipartite independent vertex sets in a grid graph.
The asymptotic behavior of their growth rates is presented.
\end{abstract}

\maketitle

\section{Introduction} \label{sec:intro}

The Merrifield--Simmons index and the Hosoya index of a graph,
respectively introduced by Merrifield and Simmons~\cite{MS1, MS2, MS3} and by Hosoya~\cite{Ho},
are two prominent examples of topological indices for the study of the relation between molecular structure 
and physical/chemical properties of certain hydrocarbon compound, 
such as the correlation with boiling points~\cite{GP}.
An {\em independent\/} set of vertices/edges of a graph $G$ is a set of which 
no two vertices of the set are connected by a single edge.
The Merrifield--Simmons index is defined as the total number, denoted by $\sigma (G)$, of independent vertex sets,
while the Hosoya index is defined as the total number of independent edge sets.
Especially, finding the Merrifield--Simmons index of graphs is known as 
the Hard Square Problem in lattice statistics.

One of important problems is to determine the extremal graphs 
with respect to these two indices within certain prescribed classes.
For example, among trees with the same number of vertices,
Prodinger and Tichy~\cite{PT} proved that 
the star maximizes the Merrifield--Simmons index, while the path minimizes it.
The situation for the Hosoya index is absolutely opposite;
the star minimizes the Hosoya index, while the path maximizes it~\cite{GP}.
A good summary of results for extremal graphs of various types can be found in a survey paper~\cite{WG}.
The interested reader is referred, however, to other articles \cite{AD, An, HIHH, XLZ, Zh, ZYAW} 
that treat several classes of graphs such as fullerene graphs, trees with prescribed degree sequence, 
graphs with connectivity at most k and the generalized Aztec diamonds.
 
We also consider a bipartite vertex set $\mathcal{V}$ in a graph $G$ in which
some vertices of $\mathcal{V}$ are colored black and the others are white. 
We say that $\mathcal{V}$ is a {\em bipartite independent vertex set\/}
if the vertices of the same color are independent
(vertices with different colors may not be independent). 
The total number of bipartite independent vertex sets in $G$ will be called 
the bipartite Merrifield--Simmons index and denoted by $\beta(G)$.
See the drawings in Figure~\ref{fig:examples} for exampels.

\begin{figure}[h]
\includegraphics{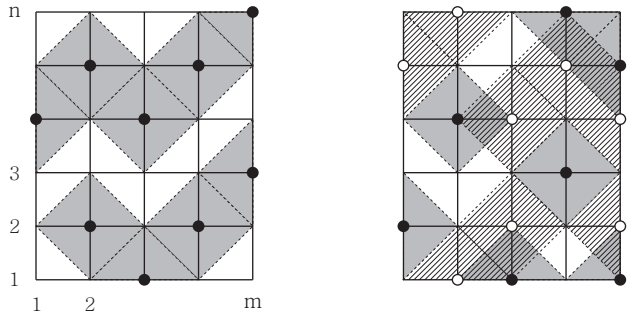}
\caption{Independent and bipartite independent vertex sets}
\label{fig:examples}
\end{figure}

Recently several important enumeration problems on two-dimensional square lattice models
have been solved by means of the {\em state matrix recursion algorithm\/}, developed by Oh in~\cite{OhD1}.
This algorithm provides recursive matrix-relations to enumerate
monomer and dimer coverings~\cite{OhD1},
multiple self-avoiding walks and polygons~\cite{OhP1},
and knot mosaics in quantum knot mosaic theory~\cite{OHLL}.
Furthermore, these recursive formulae also produce their generating functions.
Based upon these results, this algorithm shows considerable promise 
for further two-dimensional lattice model enumerations.

In this paper we use the state matrix recursion algorithm to calculate the Merrifield--Simmons index
of the $m \! \times \! n$ grid graph $G_{m \times n}$ 
and further its bipartite Merrifield--Simmons index.
Consider the generating function of independent vertex sets (IVSs) 
with variable $z$ in $G_{m \times n}$ defined by
$$ P_{m \times n}(z) \ = \ \sum k(d) \, z^d, $$
where $k(d)$ is the number of IVSs consisting of $d$ vertices.
Similarly consider the generating function for bipartite independent vertex sets (BIVSs) 
with variables $x$ and $y$ defined by
$$ Q_{m \times n}(x,y) = \sum k(c,d) \, x^c y^d, $$
where $k(c,d)$ is the number of BIVSs consisting of $c$ white vertices and $d$ black vertices.
We easily notice that $P_{m \times n}(z) = Q_{m \times n}(z,0)$.
These indices of $G_{m \times n}$ 
are then simply obtained by
$$ \sigma(G_{m \times n}) = P_{m \times n}(1) \text{ and }
\beta(G_{m \times n}) = Q_{m \times n}(1,1). $$

Hereafter $\mathbb{O}_k$ and $\mathbb{O}'_k$ denote 
the square zero-matrices of dimensions $2^k$ and $3^k$, respectively.

\begin{theorem} \label{thm:ivs}
The generating function for independent vertex sets is
\begin{equation*} \begin{aligned}
P_{m \times n}(z) & = \text{entry sum of the first column of } (A_m)^n \\
& = (1,1)\text{-entry of } (A_m)^{n+1},
\end{aligned} \end{equation*}
where $A_m$ is a $2^m \! \times \! 2^m$ matrix recursively defined by
$$A_{k+1} = \begin{bmatrix} A_k & B_k \\
z C_k & \mathbb{O}_k \end{bmatrix}\! , \
B_{k+1} = \begin{bmatrix} A_k & \mathbb{O}_k \\
z C_k &  \mathbb{O}_k \end{bmatrix} \mbox{ and \ }
C_{k+1} = \begin{bmatrix} A_k & B_k \\
\mathbb{O}_k & \mathbb{O}_k \end{bmatrix}\! ,$$
for $k=0, \dots, m \! - \! 1$, with seed matrices
$A_0 = B_0 = C_0 = \begin{bmatrix} 1 \end{bmatrix}$.
\end{theorem}

\begin{theorem} \label{thm:bivs}
The generating function for bipartite independent vertex sets is 
\begin{equation*} \begin{aligned}
Q_{m\times n}(x,y) & = \text{entry sum of the first column of }(A_m)^n \\
& = (1,1) \text{-entry of }(A_m)^{n+1},
\end{aligned} \end{equation*}
where $A_m$ is a $3^m \! \times \! 3^m$ matrix defined by
$$A_{k+1} = \begin{bmatrix} A_k & B_k & C_k \\ 
x D_k & \mathbb{O}_k & x E_k \\ y F_k & y G_k & \mathbb{O}_k \end{bmatrix}\! , \ $$
$$B_{k+1} = \begin{bmatrix} A_k & \mathbb{O}_k & C_k \\ 
x D_k & \mathbb{O}_k & x E_k \\ y F_k & \mathbb{O}_k & \mathbb{O}_k \end{bmatrix}\! , \
C_{k+1} = \begin{bmatrix} A_k & B_k & \mathbb{O}_k \\ 
x D_k & \mathbb{O}_k & \mathbb{O}_k \\ y F_k & y G_k & \mathbb{O}_k \end{bmatrix}\! , \ $$
$$D_{k+1} = \begin{bmatrix} A_k & B_k& C_k \\ 
\mathbb{O}_k & \mathbb{O}_k & \mathbb{O}_k \\ y F_k & y G_k & \mathbb{O}_k \end{bmatrix}\! , \
E_{k+1} = \begin{bmatrix} A_k & B_k& \mathbb{O}_k \\ 
\mathbb{O}_k & \mathbb{O}_k & \mathbb{O}_k \\ y F_k & y G_k & \mathbb{O}_k \end{bmatrix}\! , \ $$
$$F_{k+1} = \begin{bmatrix} A_k & B_k& C_k \\ 
x D_k & \mathbb{O}_k & x E_k \\ \mathbb{O}_k & \mathbb{O}_k & \mathbb{O}_k \end{bmatrix}\! \text{ and \ }
G_{k+1} = \begin{bmatrix} A_k & \mathbb{O}_k & C_k \\ 
x D_k & \mathbb{O}_k & x E_k \\ \mathbb{O}_k & \mathbb{O}_k & \mathbb{O}_k \end{bmatrix}\! ,$$
for $k=0,\dots,m \! - \! 1$, with seed matrices
$A_0 = \cdots = G_0 = \begin{bmatrix} 1 \end{bmatrix}$. 
\end{theorem}

As listed in Table~\ref{tab:list}, $\sigma(G_{n \times n})$, for $m \! = \! n$, 
is known as the two-dimensional Fibonacci number in virtue of 
Prodinger and Tichy's use of the Fibonacci number of graphs~\cite{PT}.
Since this sequence grows in a quadratic exponential rate, we may consider the limits
$$ \lim_{m, \, n \rightarrow \infty} (\sigma(G_{m \times n}))^{\frac{1}{mn}} = \eta  \text{ \ and } 
\lim_{m, \, n \rightarrow \infty} (\beta(G_{m \times n}))^{\frac{1}{mn}} = \kappa, $$
which are called the {\em hard square constant\/} and the {\em bipartite hard square constant\/}, respectively. 
The existence of the hard square constant was shown in~\cite{CW, We},
and the most updated estimate 
$$\eta \approx 1.5030480824753322643220663294755536893857810$$ 
appeared in~\cite{Ba}.
A two-dimensional application of the Fekete's lemma gives 
another simple proof of the existence and
mathematical lower and upper bounds for these constants.

\begin{theorem}\label{thm:growth}
The double limits $\eta$ and $\kappa$ exist.
More precisely, for any positive integers $m$ and $n$,
$$ (\sigma(G_{m \times n}))^{\frac{1}{(m+1)(n+1)}} \leq \eta 
\leq (\sigma(G_{m \times n}))^{\frac{1}{mn}}, $$
$$ (\beta(G_{m \times n}))^{\frac{1}{(m+1)(n+1)}} \leq \kappa 
\leq (\beta(G_{m \times n}))^{\frac{1}{mn}}. $$
\end{theorem}

Here we obtain $2.003942\cdots \leq \kappa \leq 2.181636\cdots $ 
by letting $m=9$ and $n=100$, computed by Matlab.

\begin{table}[h]
{\small
\begin{tabular}{clcc}      \hline \hline
 $n$ & $\sigma(G_{n \times n})$ & $(\sigma)^{\frac{1}{n^2}}$ &
 $(\sigma)^{\frac{1}{(n+1)^2}}$  \\    \hline
 1 & 2 & 2.000 & 1.189 \\
 2 & 7 & 1.627 & 1.241 \\
 3 & 63 & 1.585 & 1.296 \\
 4 & 1234 & 1.560 & 1.329 \\
 5 & 55447 & 1.548 & 1.354 \\
 6 & 5598861 & 1.540 & 1.373 \\
 7 & 1280128950 & 1.534 & 1.388 \\
 8 & 660647962955 & 1.530 & 1.399 \\
 9 & 770548397261707 & 1.527 & 1.409 \\
 10 & 2030049051145980050 & 1.524 & 1.417 \\
 11 & 12083401651433651945979 & 1.522 & 1.423 \\
 12 & 162481813349792588536582997 & 1.521 & 1.429 \\
 13 & 4935961285224791538367780371090 & 1.519 & 1.434 \\
 14 & 338752110195939290445247645371206783 & 1.518 & 1.439 \\
 15 & 52521741712869136440040654451875316861275 & 1.517 & 1.442 \\ \hline \hline
\end{tabular}
}
\vspace{4mm}
\caption{$\sigma(G_{n \times n})$ and its approximated ${\frac{1}{n^2}}$th
and ${\frac{1}{(n+1)^2}}$th powers}
\label{tab:list}
\end{table}

We adjust the main scheme of the state matrix recursion algorithm introduced in~\cite{OhD1}
to prove Theorem~\ref{thm:ivs} in Sections~\ref{sec:stage1}$\sim$\ref{sec:stage3}.

\section{Stage 1: Conversion to IVS mosaics} \label{sec:stage1}

This stage is dedicated to the installation of the mosaic system for IVSs on the grid graph.
Lomonaco and Kauffman~\cite{LK1, LK2} invented a mosaic system to give
a precise and workable definition of quantum knots representing an actual physical quantum system.
Oh {\em et al\/}. have developed a state matrix argument for the knot mosaic enumeration
in the papers \cite{HLLO, OHLL}.

This argument has been developed further into the state matrix recursion algorithm by which
we enumerate monomer--dimer coverings on the square lattice~\cite{OhD1}.
We follow the notion and terminology in~\cite{OhD1} with modification to IVSs. 
In this paper, we consider the three {\em mosaic tiles\/} $T_1$, $T_2$ and $T_3$ illustrated in Figure~\ref{fig:tile}.
Their horizontal and vertical side edges are labeled with two numbers 0, 1 and three letters a, b, c, respectively.

\begin{figure}[h]
\includegraphics{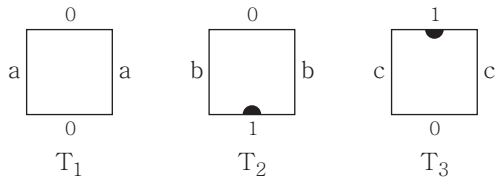}
\caption{Three mosaic tiles}
\label{fig:tile}
\end{figure}

For positive integers $m$ and $n$,
an {\em $m \! \times \! n$--mosaic\/} is an $m \! \times \! n$ rectangular array $M = (M_{ij})$ of those tiles,
where $M_{ij}$ denotes the mosaic tile placed at the $i$-th column from `left' to `right'
and the $j$-th row from `bottom' to `top'. 
We are exclusively interested in mosaics whose tiles match each other properly
to represent IVSs.
For this purpose we consider the following rules. \\

\noindent {\bf Horizontal adjacency rule:\/}
Abutting edges of adjacent mosaic tiles in a row are not labeled with any of the following pairs of letters: 
b/b, c/c. \\

\noindent {\bf Vertical adjacency rule:\/}
Abutting edges of adjacent mosaic tiles in a column must be labeled with the same number. \\

\noindent {\bf Boundary state requirement:\/}
All top boundary edges in a mosaic are labeled with number 0. (See Figure~\ref{fig:conversion}) \\

As illustrated in Figure~\ref{fig:conversion},
every IVS in $G_{m \times n}$ can be converted into 
an $m \! \times \! n$--mosaic which satisfies the three rules.
In this mosaic, 
two $T_2$'s (similarly $T_3$'s) cannot be placed adjacently in a row (horizontal adjacency rule),
while $T_2$ and $T_3$ can be adjoined along the edges labeled with number 1 (vertical adjacency rule).

\begin{figure}[h]
\includegraphics{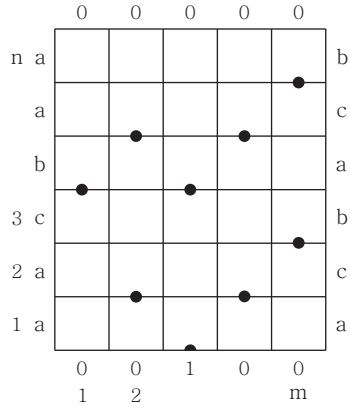}
\caption{Conversion of the IVS in Figure~\ref{fig:examples} to an IVS $m \! \times \! n$--mosaic}
\label{fig:conversion}
\end{figure}

A mosaic is said to be {\em suitably adjacent\/} if any pair of mosaic tiles
sharing an edge satisfies both adjacency rules.
A suitably adjacent $m \! \times \! n$--mosaic is called an {\em IVS $m \! \times \! n$--mosaic\/}
if it additionally satisfies the boundary state requirement.
The following one-to-one conversion arises naturally. \\

\noindent {\bf One-to-one conversion:\/}
There is a one-to-one correspondence between  
IVSs in $G_{m \times n}$ and IVS $m \! \times \! n$--mosaics.
Furthermore, the number of vertices in an IVS is equal to
the number of $T_2$ mosaic tiles in the corresponding IVS $m \! \times \! n$--mosaic.

\section{Stage 2: State matrix recursion formula} \label{sec:stage2}

Now we introduce two types of state matrices for suitably adjacent mosaics.

\subsection{States and state polynomials}
A {\em state\/} is a finite sequence of two numbers 0 and 1, or three letters a, b and c.
Let $p \leq m$ and $q \leq n$ be positive integers, 
and consider a suitably adjacent $p \! \times \! q$--mosaic $M$.
We use $d(M)$ to denote the number of appearances of $T_2$ tiles in $M$.
The {\em $b$--state\/} $s_b(M)$ ({\em $t$--state\/} $s_t(M)$) is 
the state of length $p$ obtained by reading off numbers on the bottom (top, respectively) 
boundary edges from right to left, 
and the {\em $l$--state\/} $s_l(M)$ ({\em $r$--state\/} $s_r(M)$) is 
the state of length $q$ obtained by reading off letters on the left (right, respectively) 
boundary edges from top to bottom as shown in Figure~\ref{fig:arrow}.

\begin{figure}[h]
\includegraphics{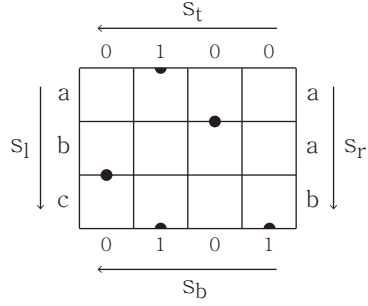}
\caption{A suitably adjacent $4 \! \times \! 3$--mosaic with four state indications:
$s_b(M) =$ 1010, $s_t(M) =$ 0010, $s_l(M) =$ abc, and $s_r(M) =$ aab}
\label{fig:arrow}
\end{figure}

Given a triple $\langle s_r, s_b, s_t \rangle$ of $r$--, $b$-- and $t$--states,
we associate the {\em state polynomial\/}:
$$ S_{\langle s_r, s_b, s_t  \rangle}(z) = \sum k(d) z^d, $$
where $k(d)$ equals the number of all suitably adjacent $p \! \times \! q$--mosaics $M$ such that 
$d(M)=d$, $s_r(M) = s_r$, $s_b(M) = s_b$ and $s_t(M) = s_t$.
Note that there is no restriction on the $l$--state of $M$.

\subsection{Bar state matrices}
Now consider suitably adjacent $p \! \times \! 1$--mosaics, which are called {\em bar mosaics\/}.
Bar mosaics of length $p$ have $2^p$ kinds of $b$-- and $t$--states, especially called {\em bar states\/}.
We arrange all bar states, which are binary digits, as usual.
For $1 \leq i \leq 2^p$, let $\epsilon^p_i$ denote the $i$-th bar state of length $p$.
The first bar state $\epsilon^p_1=$ 00$\cdots$0 is called {\em trivial\/}.

{\em Bar state matrix\/} $X_p$ ($X = A, B, C$)
for the set of suitably adjacent bar mosaics of length $p$ is a $2^p \! \times \! 2^p$ matrix $(x_{ij})$ given by  
$$ x_{ij} = S_{\langle \text{x}, \epsilon^p_i, \epsilon^p_j \rangle}(z), $$
where x $=$ a, b, c, respectively.
We remark that information on suitably adjacent bar mosaics is completely encoded 
in three bar state matrices $A_p$, $B_p$ and $C_p$.
 
\begin{lemma} [Bar state matrix recursion lemma] \label{lem:bar}
Bar state matrices $A_p$, $B_p$ and $C_p$ are recursively obtained by 
$$A_{k+1} = \begin{bmatrix} A_k \! + \! B_k \! + \! C_k &  \mathbb{O}_k \\
 \mathbb{O}_k & \mathbb{O}_k \end{bmatrix},$$
$$B_{k+1} = \begin{bmatrix} \mathbb{O}_k & \mathbb{O}_k \\
z \, A_k \! + \! z \, C_k &  \mathbb{O}_k \end{bmatrix} \mbox{ and\,  }
C_{k+1} = \begin{bmatrix} \mathbb{O}_k & A_k \! + \! B_k \\
 \mathbb{O}_k & \mathbb{O}_k \end{bmatrix}$$
with seed matrices
$$A_1 = \begin{bmatrix} 1 & 0 \\ 0 & 0 \end{bmatrix}, \
B_1 = \begin{bmatrix} 0 & 0 \\ z & 0 \end{bmatrix} \mbox{ and\, }
C_1 = \begin{bmatrix} 0 & 1 \\ 0 & 0 \end{bmatrix}.$$
\end{lemma}

Note that we may start with matrices
$A_0 = \begin{bmatrix} 1 \end{bmatrix}$ and $B_0 = C_0 = \begin{bmatrix} 0 \end{bmatrix}$
instead of $A_1$, $B_1$ and $C_1$.
Our proofs of Lemmas~\ref{lem:bar} and \ref{lem:mn} parallel  respectively those of Lemmas~5 and 6
in~\cite{OhD1} with slight modification.

\begin{proof}
We use induction on $k$.
A straightforward observation on the mosaic tiles establishes the lemma for $k=1$.

Assume that bar state matrices $A_k$, $B_k$ and $C_k$ satisfy the statement.
Consider the matrix $B_{k+1}$, which is of size $2^{k+1} \! \times \! 2^{k+1}$.
Partition this matrix into four block submatrices of size $2^k \! \times \! 2^k$, 
and consider the 21-submatrix of $B_{k+1}$,
i.e., the $(2,1)$-component in the $2 \! \times \! 2$ array of the four blocks.
The $(i,j)$-entry of the 21-submatrix is the state polynomial 
$S_{\langle \text{b}, \text{1}\epsilon^k_i, \text{0}\epsilon^k_j \rangle}(z)$
where 1$\epsilon^k_i$ (similarly 0$\epsilon^k_j$) is a bar state of length $k \! + \! 1$
obtained by concatenating two bar states 1 and $\epsilon^k_i$.
A suitably adjacent $(k \! + \! 1) \! \times \! 1$--mosaic corresponding to 
this triple $\langle \text{b}, \text{1}\epsilon^k_i, \text{0}\epsilon^k_j \rangle$
must have tile $T_2$ at the place of the rightmost mosaic tile, 
and so its second rightmost tile cannot be $T_2$ by the horizontal adjacency rule.
Thus the $r$--state of the second rightmost tile is either a or c.
By considering the contribution of the rightmost tile $T_2$ to the state polynomial,
one easily gets
$$S_{\langle \text{b}, \text{1}\epsilon^k_i, \text{0}\epsilon^k_j \rangle}(z) = 
z \cdot \big( (i,j)\text{-entry of } (A_k + C_k) \big).$$
Thus the 21-submatrix of $B_{k+1}$ is $z A_k + z C_k$.
The same argument gives Table~\ref{tab:barset} presenting all possible twelve cases as desired.
\end{proof}

\begin{table}[h]
\begin{tabular}{cccc}      \hline \hline
  & \ {\em Submatrix for\/} $\langle s_r, s_b, s_t \rangle$ \ & \ {\em Rightmost tile\/} \ &
{\em Submatrix\/} \\    \hline
\ $A_{k+1}$ \ & 11-submatrix $\langle \text{a}, \text{0} \! \cdot \! \cdot,\text{0} \! \cdot \! \cdot \rangle$ 
& $T_1$ & \ $A_k \! + \! B_k \! + \! C_k$ \ \\
$B_{k+1}$ & 21-submatrix $\langle \text{b}, \text{1} \! \cdot \! \cdot,\text{0} \! \cdot \! \cdot \rangle$ 
& $T_2$ & $z \, A_k \! + \! z \, C_k$ \\
$C_ {k+1}$ & 12-submatrix $\langle \text{c}, \text{0} \! \cdot \! \cdot,\text{1} \! \cdot \! \cdot \rangle$
& $T_3$ & $A_k \! + \! B_k$ \\
 & The other nine cases & None & $\mathbb{O}_k$ \\     \hline \hline
\end{tabular}
\vspace{4mm}
\caption{Twelve submatrices of $A_{k+1}$, $B_{k+1}$ and $C_{k+1}$}
\label{tab:barset}
\end{table}

\subsection{State matrices}

{\em State matrix\/} $Y_{m \times q}$ for the set of suitably adjacent $m \! \times \! q$--mosaics  
is a $2^m \! \times \! 2^m$ matrix $(y_{ij})$ given by 
$$ y_{ij} = \sum S_{\langle s_r, \epsilon^m_i, \epsilon^m_j \rangle}(z), $$
where the summation is taken over all $r$--states $s_r$ of length $q$.

\begin{lemma} [State matrix multiplication lemma] \label{lem:mn}
$$Y_{m \times n} = (A_m + B_m + C_m)^n.$$
\end{lemma}

\begin{proof}
Use induction on $n$.
For $n=1$, $Y_{m \times 1} =A_m + B_m + C_m$
since $Y_{m \times 1}$ counts suitably adjacent $m \! \times \! 1$--mosaics with any $r$--states.
Assume that $Y_{m \times k} = (A_m + B_m + C_m)^k$.
Consider a suitably adjacent $m \! \times \! (k \! + \! 1)$--mosaic $M^{m \times (k+1)}$.
Split it into two suitably adjacent $m \! \times \! k$-- and $m \! \times \! 1$--mosaics
$M^{m \times k}$ and $M^{m \times 1}$ by tearing off the topmost bar mosaic.
By the vertical adjacency rule,
the $t$--state of $M^{m \times k}$ and the $b$--state of $M^{m \times 1}$ must coincide
as shown in Figure~\ref{fig:expand}.

\begin{figure}[h]
\includegraphics{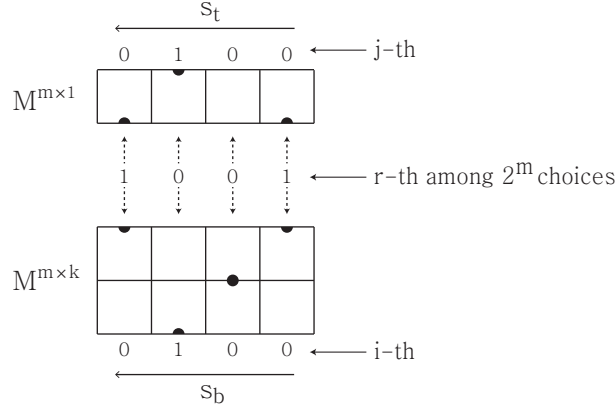}
\caption{Expanding $M^{m \times k}$ to $M^{m \times (k+1)}$}
\label{fig:expand}
\end{figure}

Let $Y_{m \times (k+1)} = (y_{ij})$, $Y_{m \times k} = (y_{ij}')$ and $Y_{m \times 1} = (y_{ij}'')$.
Note that $y_{ij}$ is the state polynomial 
for the set of suitably adjacent $m \! \times \! (k \! + \! 1)$--mosaics $M$ 
which admit splittings into $M^{m \times k}$ and $M^{m \times 1}$ satisfying
$s_b(M) = s_b(M^{m \times k}) = \epsilon^m_i$,
$s_t(M) = s_t(M^{m \times 1}) = \epsilon^m_j$, and
$s_t(M^{m \times k}) = s_b(M^{m \times 1}) = \epsilon^m_r$ ($1 \leq r \leq 2^m$).
Thus, 
$$ y_{ij} = \sum^{2^m}_{r=1} y_{ir}'  \cdot  y_{rj}''.$$
This implies
$$ Y_{m \times (k+1)} = Y_{m \times k} \cdot Y_{m \times 1} = (A_m + B_m + C_m)^{k+1},$$
and the induction step is finished.
\end{proof}

\section{Stage 3: State matrix analyzing} \label{sec:stage3}

We analyze state matrix $Y_{m \times n}$ to find the generating function $P_{m \times n}(z)$.

\begin{proof}[Proof of Theorem~\ref{thm:ivs}.]
The $(i,j)$-entry of $Y_{m \times n}$ is the state polynomial for the set of 
suitably adjacent $m \! \times \! n$--mosaics $M$ with
$s_b(M) = \epsilon^m_i$ and $s_t(M) = \epsilon^m_j$ (no restriction on $s_l(M)$ and $s_r(M)$).
According to the boundary state requirement,
IVSs in $G_{m \times n}$ are converted into 
suitably adjacent $m \! \times \! n$--mosaics $M$ with trivial $t$--state
as the left picture in Figure~\ref{fig:analyze}.
This means $s_b(M) = \epsilon^m_i$ ($i$ takes any value of $1,\dots,2^m$) and $s_t(M) = \epsilon^m_1$.
Thus the sum of the state polynomials in the first column of $Y_{m \times n}$ represents
the generating function $P_{m \times n}(z)$.
In short, we get 
$$P_{m \times n}(z) = \mbox{entry sum of the first column of } Y_{m \times n}.$$

On the other hand, as the right picture in Figure~\ref{fig:analyze},
IVS $m \! \times \! n$--mosaics can also be converted to
suitably adjacent $m \! \times \! (n \! + \! 1)$--mosaics with trivial $b$-- and $t$--states.
Therefore,
$$P_{m \times n}(z) = \mbox{(1,1)-entry of } Y_{m \times (n+1)}.$$
These equalities combined with Lemmas~\ref{lem:bar} and~\ref{lem:mn}
complete the proof.

Note that the recurrence relation in Lemma~\ref{lem:bar}
is easily translated into that of Theorem~\ref{thm:ivs} by replacing 
$A_k \! + \! B_k \! + \! C_k$, $A_k \! + \! B_k$ and $A_k \! + \! C_k$ 
with $A_k$, $B_k$ and $C_k$, respectively.
\end{proof}

\begin{figure}[h]
\includegraphics{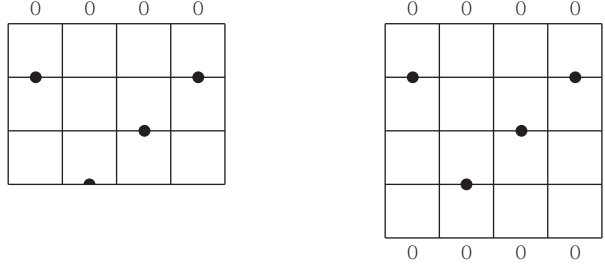}
\caption{Analyzing state matrix $Y_{m \times n}$}
\label{fig:analyze}
\end{figure}

\section{BIVS mosaics} \label{sec:bivs}
 
In this section we use the state matrix recursion algorithm to enumerate 
bipartite independent vertex sets.
We follow the argument in the proof of Theorem~\ref{thm:ivs}.

\begin{proof}[Proof of Theorem~\ref{thm:bivs}.]
We reformulate the state matrix recursion algorithm 
by using seven mosaic tiles $T_1, \dots, T_7$ illustrated in Figure~\ref{fig:tile2}.
Their horizontal and vertical side edges are labeled with three numbers 0, 1, 2 and 
seven letters a, b, c, d, e, f, g, respectively. 

\begin{figure}[h]
\includegraphics{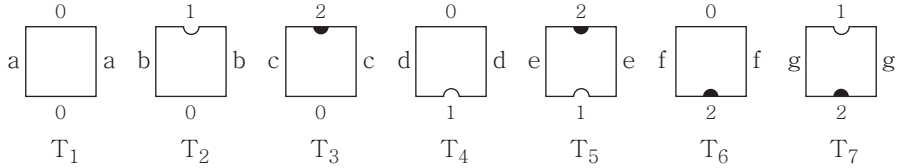}
\caption{Seven mosaic tiles}
\label{fig:tile2}
\end{figure}

The same vertical adjacency rule and boundary state requirement are employed,
while the horizontal adjacency rule and the corresponding one-to-one conversion are slightly changed
as follows. \\

\noindent {\bf Horizontal adjacency rule:\/} 
Abutting edges of adjacent mosaic tiles in a row are not labeled with any of the following pairs of letters:
b/b, c/c, d/d, e/e, f/f, g/g, b/g, g/b, c/e, e/c, d/e, e/d, f/g, g/f. \\

\noindent {\bf One-to-one conversion:\/}
There is a one-to-one correspondence between  
BIVSs in $G_{m \times n}$ and BIVS $m \! \times \! n$--mosaics.
Furthermore, the number of white (black) vertices in a BIVS is equal to 
the number of $T_4$ and $T_5$ ($T_6$ and $T_7$, respectively) mosaic tiles
in the corresponding BIVS $m \! \times \! n$--mosaic. \\

In the second stage, we find the corresponding bar state matrix recursion lemma (Lemma~\ref{lem:bar})
and state matrix multiplication lemma (Lemma~\ref{lem:mn}) as in Section~\ref{sec:stage2}.

\begin{lemma} \label{lem:bbar}
Bar state matrices $A_p, \dots, G_p$ are obtained by the recurrence relations:
\begin{equation*}
\begin{split}
A_{k+1} & = A_1 \otimes (A_k \! + \! B_k \! + \! C_k \! + \! D_k \! + \! E_k \! + \! F_k \! + \! G_k) \\
B_{k+1} & = B_1 \otimes (A_k \! + \! C_k \! + \! D_k \! + \! E_k \! + \! F_k) \\
C_{k+1} & = C_1 \otimes (A_k \! + \! B_k \! + \! D_k \! + \! F_k \! + \! G_k) \\
D_{k+1} & = D_1 \otimes (A_k \! + \! B_k \! + \! C_k \! + \! F_k \! + \! G_k) \\
E_{k+1} & = E_1 \otimes (A_k \! + \! B_k \! + \! F_k \! + \! G_k) \\
F_{k+1} & = F_1 \otimes (A_k \! + \! B_k \! + \! C_k \! + \! D_k \! + \! E_k) \\
G_{k+1}&  = G_1 \otimes (A_k \! + \! C_k \! + \! D_k \! + \! E_k)
\end{split}
\end{equation*}
with seed matrices
$$A_1=\begin{bmatrix} 1 & 0 & 0 \\ 0 & 0 & 0 \\ 0 & 0 & 0 \end{bmatrix}\! ,
B_1=\begin{bmatrix} 0 & 1 & 0 \\ 0 & 0 & 0 \\ 0 & 0 & 0 \end{bmatrix}\! ,
C_1=\begin{bmatrix} 0 & 0 & 1 \\ 0 & 0 & 0 \\ 0 & 0 & 0 \end{bmatrix}\! ,
D_1=\begin{bmatrix} 0 & 0 & 0 \\ x & 0 & 0 \\ 0 & 0 & 0 \end{bmatrix}\! ,$$
$$E_1=\begin{bmatrix} 0 & 0 & 0 \\ 0 & 0 & x \\ 0 & 0 & 0 \end{bmatrix}\! ,
F_1=\begin{bmatrix} 0 & 0 & 0 \\ 0 & 0 & 0 \\ y & 0 & 0 \end{bmatrix} \text{and }
G_1=\begin{bmatrix} 0 & 0 & 0 \\ 0 & 0 & 0 \\ 0 & y & 0 \end{bmatrix}\! .$$
\end{lemma}

\begin{lemma} \label{lem:bmn}
$$Y_{m \times n} = (A_m + B_m + C_m + D_m + E_m + F_m + G_m)^n.$$
\end{lemma}

In the third stage, we analyze this state matrix as in Section~\ref{sec:stage3}, and
as done there, we replace
$A_k \! + \! B_k \! + \! C_k \! + \! D_k \! + \! E_k \! + \! F_k \! + \! G_k$, 
$A_k \! + \! C_k \! + \! D_k \! + \! E_k \! + \! F_k$,
$A_k \! + \! B_k \! + \! D_k \! + \! F_k \! + \! G_k$,
$A_k \! + \! B_k \! + \! C_k \! + \! F_k \! + \! G_k$,
$A_k \! + \! B_k \! + \! F_k \! + \! G_k$,
$A_k \! + \! B_k \! + \! C_k \! + \! D_k \! + \! E_k$,
$A_k \! + \! C_k \! + \! D_k \! + \! E_k$ 
with $A_k, \dots, G_k$, respectively,
to complete the proof.
\end{proof}

\section{Hard square constant} \label{sec:growth}

To prove Theorem~\ref{thm:growth}, 
we need the following result called Fekete's lemma with slight modification.

\begin{lemma}{\cite[Lemma~7]{OhD1}} \label{lem:growth}
Suppose that a double sequence $\{ a_{m,n} \}_{m, \, n \in \, \mathbb{N}}$ with $a_{m,n} \geq 1$ satisfies
$a_{m_1+m_2,n} \leq a_{m_1,n} \cdot a_{m_2,n} \leq a_{m_1 + m_2 + 1,n}$
and $a_{m,n_1+n_2} \leq a_{m,n_1} \cdot a_{m,n_2} \leq a_{m,n_1 + n_2 + 1}$
for all $m$, $m_1$, $m_2$, $n$, $n_1$ and $n_2$.
Then
$$ \lim_{m, \, n \rightarrow \infty} (a_{m,n})^{\frac{1}{mn}} = 
\inf_{m, \, n \in \, \mathbb{N}} (a_{m,n})^{\frac{1}{mn}} = 
\sup_{m, \, n \in \, \mathbb{N}} (a_{m,n})^{\frac{1}{(m+1)(n+1)}},$$
provided that the supremum exists.
\end{lemma}

\begin{proof}[Proof of Theorem~\ref{thm:growth}.]
Consider the Merrifield--Simmons index $\sigma(G_{m \times n})$, 
simply denoted by $\sigma_{m \times n}$.
Obviously, $\sigma_{m \times n} \geq 1$ for all $m,n$.
The submultiplicative inequality
$\sigma_{(m_1+m_2) \times n} \leq \sigma_{m_1 \times n} \cdot \sigma_{m_2 \times n}$ 
is obvious because we can always split an IVS $(m_1 \! + \! m_2) \! \times \! n$--mosaic
into a unique pair of IVS $m_1 \! \times \! n$-- and $m_2 \! \times \! n$--mosaics.
On the other hand,
any two IVS $m_1 \! \times \! n$-- and $m_2 \! \times \! n$--mosaics
can be adjoined horizontally to create a new IVS $(m_1 \! + \! m_2 \! + \! 1) \! \times \! n$--mosaic
by inserting between them a $1 \! \times \! n$--mosaic consisting only of $T_1$ tiles as in Figure~\ref{fig:growth}.
Therefore $\sigma_{m_1 \times n} \cdot \sigma_{m_2 \times n} \leq \sigma_{(m_1+m_2+1) \times n}$.

The inequality 
$\sigma_{m \times (n_1+n_2)} \leq \sigma_{m \times n_1} \cdot \sigma_{m \times n_2}$ is also obvious 
because we can always split an IVS $m \! \times \! (n_1 \! + \! n_2)$--mosaic
into a unique pair of IVS $m \! \times \! n_1$-- and $m \! \times \! n_2$--mosaics
by deleting all vertices on the top boundary of the bottom-side $m \! \times \! n_1$--mosaic.
On the other hand,
any two IVS $m \! \times \! n_1$-- and $m \! \times \! n_2$--mosaics
$M^{m \times n_1}$ and $M^{m \times n_2}$
can be adjoined vertically to create a new IVS $m \! \times \! (n_1 \! + \! n_2 \! + \! 1)$--mosaic
by inserting a suitably adjacent bar $m \! \times \! 1$--mosaic 
whose $b$--state is trivial as $s_t(M^{m \times n_1})$ and $t$--state is $s_b(M^{m \times n_2})$ 
as in Figure~\ref{fig:growth}.
Therefore $\sigma_{m \times n_1} \cdot \sigma_{m \times n_2} \leq \sigma_{m \times (n_1+n_2+1)}$.
Since we use only three mosaic tiles at each site,
$\sup_{m, \, n} (\sigma_{m \times n})^{\frac{1}{(m+1)(n+1)}} \leq 3$, and now apply Lemma~\ref{lem:growth}.

For the bipartite Merrifield--Simmons index $\beta(G_{m \times n})$, this proof applies verbatim.
\end{proof}

\begin{figure}[h]
\includegraphics{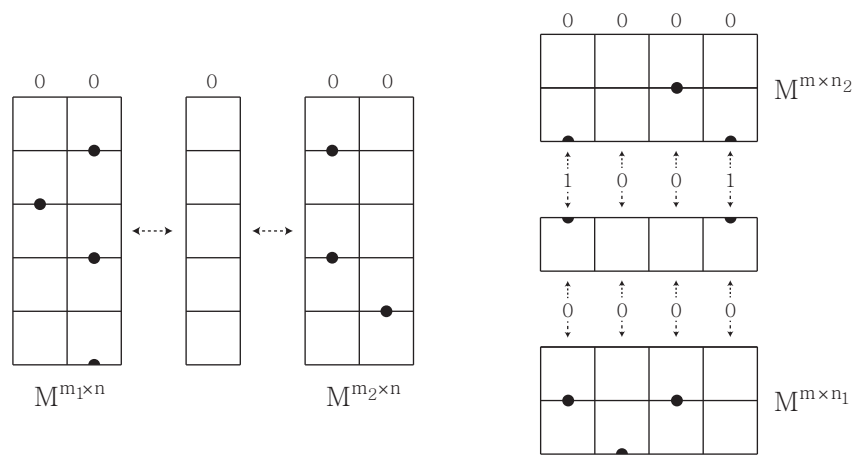}
\caption{Adjoining two IVS mosaics}
\label{fig:growth}
\end{figure}


\begin{thebibliography}{11}

\bibitem{AD} M. Ahmadi and H. Dastkhezr,
    {\em An algorithm for computing the Merrifield-Simmons Index},
    MATCH Commun. Math. Comput. Chem. \textbf{71} (2014) 355--359.  
\bibitem{An} E. Andriantiana,
    {\em Energy, Hosoya index and Merrifield-Simmons index of trees with prescribed degree sequence},
    Discrete Appl. Math. \textbf{161} (2013) 724--741.  
\bibitem{Ba} R. Baxter,
    {\em Planar lattice gases with nearest-neighbor exclusion},
    Ann. Comb. \textbf{3} (1999) 191--203.
\bibitem{CW} N. Calkin and H. Wilf,
    {\em The number of independent sets in a grid graph},
    SIAM J. Discrete Math. \textbf{11} (1998) 54--60.
\bibitem{GP} I. Gutman and O. Polansky,
    Mathematical Concepts in Organic Chemistry (Springer, Berlin) (1986).
\bibitem{HIHH} A. Hamzeh, A. Iranmanesh, S. Hossein-Zadeh and M. Hosseinzadeh,
    {\em The Hosoya index and the Merrifield-Simmons index of some graphs},
    Trans. Comb. \textbf{1} (2012) 51--60.
\bibitem{HLLO} K. Hong, H. Lee, H.J. Lee and S. Oh,
    {\em Small knot mosaics and partition matrices},
    J. Phys. A: Math. Theor. \textbf{47} (2014) 435201.
\bibitem{Ho} H. Hosoya,
    {\em Topological Index. A newly proposed quantity characterizing the topological nature of structural isomers
    of saturated hydrocarbons},
    Bull. Chem. Soc. Jpn. \textbf{44} (1971) 2332--2339.
\bibitem{LK1} S. Lomonaco and L. Kauffman,
    {\em Quantum knots},
    Quantum Information and Computation II, Proc. SPIE \textbf{5436} (2004) 268--284.
\bibitem{LK2} S. Lomonaco and L. Kauffman,
    {\em Quantum knots and mosaics},
    Quantum Inf. Process. \textbf{7} (2008) 85--115.
\bibitem{MS1} R. Merrifield and H. Simmons,
    {\em Enumeration of structure-sensitive graphical subsets: theory},
    Proc. Natl. Acad. Sci. USA \textbf{78} (1981) 692--695.
\bibitem{MS2} R. Merrifield and H. Simmons,
    {\em Enumeration of structure-sensitive graphical subsets: calculations},
    Proc. Natl. Acad. Sci. USA \textbf{78} (1981) 1329--1332.
\bibitem{MS3} R. Merrifield and H. Simmons,
    Topological Methods in Chemistry (Wiley, New York) (1989).
\bibitem{OhD1} S. Oh,
    {\em State matrix recursion algorithm and monomer--dimer problem},
    (preprint).
\bibitem{OhP1} S. Oh and K. Hong,
    {\em Multiple self-avoiding walk and polygon enumeration by state matrix recursion algorithm},
    (preprint).
\bibitem{OHLL} S. Oh, K. Hong, H. Lee and H. J. Lee,
    {\em Quantum knots and the number of knot mosaics},
    Quantum Inf. Process. \textbf{14} (2015) 801--811.
\bibitem{PT} H. Prodinger and R. Tichy,
    {\em Fibonacci numbers of graphs},
    Fibonacci Quart. \textbf{20} (1982) 16--21.
\bibitem{WG} S. Wagner and I. Gutman,
    {\em Maxima and minima of the Hosoya index and the Merrifield--Simmons index:
    a survey of results and techniques},
    Acta Appl. Math. \textbf{112} (2010) 323--346.
\bibitem{We} K. Weber,
    {\em On the number of stable sets in an $m \! \times \! n$ lattice},
    Rostock. Math. Kolloq. \textbf{34} (1988) 28--36.
\bibitem{XLZ} K. Xu, J. Li and L. Zhong,
    {\em The Hosoya indices and Merrifield-Simmons indices of graphs with connectivity at most $k$},
    Appl. Math. Lett. \textbf{25} (2012) 476--480.
\bibitem{Zh} Z. Zhang,
    {\em Merrifield-Simmons index of generalized Aztec diamonds and related graphs},
    MATCH Commun. Math. Comput. Chem. \textbf{56} (2006) 625--636.
\bibitem{ZYAW} Z. Zhu, C. Yuan, E. Andriantiana and S. Wagner,
    {\em Graphs with maximal Hosoya index and minimal Merrifield-Simmons index},
    Discrete Math. \textbf{329} (2014) 77--87.

\end{thebibliography}
\end{document}